\documentclass[dvi]{imsart}

\RequirePackage[OT1]{fontenc}
\RequirePackage{amsthm,amsmath}
\RequirePackage{amssymb}
\RequirePackage[numbers]{natbib}
\RequirePackage[colorlinks,citecolor=blue,urlcolor=blue]{hyperref}

\startlocaldefs
\numberwithin{equation}{section}
\theoremstyle{plain}
\newtheorem{thm}{Theorem}[section]
\theoremstyle{remark}

\theoremstyle{plain}
\newtheorem{prop}{Proposition}[section]
\newtheorem{coroll}{Corollary}[section]

\endlocaldefs

\begin{document}

\begin{frontmatter}
\title{On the convergence of the Metropolis-Hastings Markov chains}
\runtitle{On the convergence of the M-H Markov chains}

\begin{aug}
\ead[label=u1,url]{http://arxiv.org/}
\author{\fnms{Dimiter} \snm{Tsvetkov}
\thanksref{a}
\ead[label=e1]{dimiter.tsvetkov@yahoo.com}}
\author{\fnms{Lyubomir} \snm{Hristov}
\thanksref{a}
\ead[label=e2]{lyubomir.hristov@gmail.com}}
\and
\author{\fnms{Ralitsa} \snm{Angelova-Slavova}
\thanksref{b}
\ead[label=e3]{ralitsa.slavova@yahoo.com}
}

\address[a]{
Department of "Mathematical Analysis and Applications" \\ Faculty of Mathematics and Informatics
\\ St. Cyril and St. Methodius University of Veliko Tarnovo
\printead{e1,e2}}

\address[b]{
Department of "Communication Systems and Technologies" \\ Vasil Levski National Military University at Veliko Tarnovo
\printead{e3}
}

\runauthor{D. Tsvetkov et al.}


\end{aug}

\begin{abstract}
In this paper we study Markov chains associated with the Metropolis-Hastings algorithm.
We consider conditions under which the sequence of the successive densities of such a chain converges to the
target density according to the total variation distance for any choice of the initial density.
In particular we prove that the positiveness of the proposal density is enough for the chain to converge.
The content of this work basically presents a stand alone proof that the reversibility along with the kernel positivity imply the convergence.
\end{abstract}



\end{frontmatter}

\section{Introduction and main result}
The Metropolis-Hastings algorithm invented by Nicholas Metropolis at al. \cite{origin-2} and W. Keith Hastings \cite{origin-1} is one of the best recognized
Markov chain Monte Carlo (MCMC) techniques
in the statistical applications (see e.g. \cite{basic-1,book-13,advanced-1,book-14,book-9,basic-3,book-12,paper-1,book-16}).
Throughout this paper we shall assume that the following conditions are valid.

{\color{blue} $\mathcal{H}$:}
Let $(\mathbb{X},\mathcal{A},\lambda)$ be some measure space with a $\sigma$-finite measure $\lambda$.
Assume we are given a \textit{target} probability distribution on $(\mathbb{X},\mathcal{A})$ which is absolutely continuous with respect to $\lambda$
with density $\pi(\cdot):\mathbb{X}\rightarrow\mathbb{R}_{+}$ for which $\pi(x)>0$ for all $x\in\mathbb{X}$.
Assume also we are given an absolutely continuous with respect to $\lambda$ \textit{proposal} distribution
on $(\mathbb{X},\mathcal{A})$ which density $q(\cdot|x):\mathbb{X}\rightarrow\mathbb{R}_{+}$ is set conditionally to $x\in\mathbb{X}$.
It is assumed that $q(\cdot|\cdot):\mathbb{X}\times\mathbb{X}\rightarrow \mathbb{R}_{+}$
is jointly $\mathcal{A}\times\mathcal{A}$ measurable (see e.g. \cite{basic-2}). $\Box$

Note that {\color{blue} $\mathcal{H}$} brings entirely technical nature, providing the proper container for the further treatments.
The assumption that $\pi(x)>0$ for all $x\in\mathbb{X}$
offers also some technical facilities and does not represent a limitation of the generality in the present work purposes.

All the probability densities in this paper are considered with respect to the common reference measure $\lambda$.

The Metropolis-Hastings algorithm, as a MCMC algorithm, serves for sampling from the target distribution $\pi(\cdot)$ and consists of the following steps.
Generate first initial draw $x_{(0)}$. Let we know the current draw $x_{(n-1)}$.
To obtain the next draw $x_{(n)}$ one should generate a candidate $x_{\ast}\sim q(x|x_{(n-1)})$ and accept the candidate
with a probability $\alpha(x_{(n-1)},x_{\ast})$ taking $x_{(n)} = x_{\ast}$ or reject the candidate with a
probability $1-\alpha(x_{(n-1)},x_{\ast})$ and take $x_{(n)} = x_{(n-1)}$ where
\begin{equation*}
    \alpha(x,x') = \min \left( 1 , \frac{\pi(x')}{\pi(x)}  \frac{q(x|x')}{q(x'|x)} \right)  \text{for }  \pi(x)q(x'|x) > 0.
\end{equation*}
Following \cite{paper-1} we set $\alpha(x,x')=1$ for $q(x'|x) = 0$ to avoid ambiguity. All draws are taken from $\mathbb{X}$.
This scheme defines a transition kernel
\begin{multline}\label{e1-1}
    \kappa (x\rightarrow x') = \alpha(x,x') q(x'|x) + \delta(x-x')\int \left( 1 - \alpha(x,z) \right) q(z|x) \lambda(dz).
\end{multline}
where $\delta(\cdot)$ is the delta function.
The integral sign stands for the $\lambda$ integration over $\mathbb{X}$ (including where it is necessary the delta function rule).
The notation $\kappa(x\rightarrow x')$ stands for a function of two variables $(x,x')\in\mathbb{X}\times\mathbb{X}$
associated (by analogy to the discrete state space) with the conditional probability to move from state $x$ to state $x'$.
According to the assumptions for $\pi(\cdot)$ and $q(\cdot|\cdot)$ the kernel (\ref{e1-1}) is nonnegative function.
This kernel fulfills the normalizing condition $\int \kappa (x\rightarrow x')\lambda(dx') = 1$
but first of all it is well known that the kernel satisfies the \textit{detailed balance condition} (reversibility of the chain)
\begin{equation}\label{e1-2}
    \pi (x)\kappa(x\rightarrow x') = \pi (x')\kappa(x'\rightarrow x).
\end{equation}

From the detailed balance condition it follows that the target density is
\textit{invariant} (stationary) for the kernel, i.e. it holds $\pi(x') = \int \pi(x)\kappa(x\rightarrow x')\lambda(dx).$
The transition kernel (\ref{e1-1}) defines a Metropolis-Hastings Markov chain (shortly MH-chain) of $\mathbb{X}$-valued random variables $\left( X_{(n)}\right)$
according to the following rule. Define the initial random variable $X_{(0)}$ with some proper
density $f_{(0)}(\cdot):\mathbb{X}\rightarrow\mathbb{R}_{+}$.
For any next random variable $X_{(n)}$ the corresponding density $f_{(n)}(\cdot)$  is defined by the recurrent formula
\begin{equation*}
    f_{(n)}(x') = \int f_{(n-1)}(x)\kappa(x\rightarrow x')\lambda(dx), n=1,2,\ldots .
\end{equation*}

One of the main problems arise here is to establish conditions under which the sequence $(f_{(n)}(\cdot))$ converges to the
invariant density $\pi(\cdot)$. In the general case of stationary Markov chain usually proves that this
sequence converges with respect to the total variation distance $d_{TV}$, i.e. that
\begin{equation}\label{e1-3}
    \lim_{n\rightarrow\infty} d_{TV}(\mu[f_{(n)}], \mu[\pi])  = \lim_{n\rightarrow\infty} \frac{1}{2}\int | f_{(n)}(x) - \pi(x) | \lambda(dx) = 0
\end{equation}
under various specific assumptions (see e.g. \cite{book-1, basic-1, basic-2, paper-2, advanced-2, book-7, basic-3, paper-1}).
Here by $\mu[f]$ we denote the probability measure associated with the density $f(\cdot)$.

In this paper (Theorem \ref{t31}) we propose conditions under which (\ref{e1-3}) holds
but we follow a somewhat different approach
by means of the properly defined Hilbert space described for example in Stroock \cite{origin-3}.


\section{Some preliminaries and notes}
This study was motivated by the well-known and clearly proven result in the discrete case where the simple positivity of the transition matrix (or some its power)
guarantees the convergence of the corresponding Markov chain to the stationary distribution, i.e. the positiveness occurs the only constructive condition needed
for the convergence. So a natural question arises whether some proper positiveness condition is also sufficient for the convergence in the general state discrete
time Markov chains? The answer turns out to be positive in the case of the Markov chains, associated with the Metropolis-Hastings algorithm.

	The general state discrete time Markov chains convergence is well investigated (see e.g. \cite{book-1, basic-1, basic-2, paper-2, advanced-2, book-7, basic-3, paper-1})
and very common advanced results were achieved by using of some specific notions as reversibility, irreducibility and aperiodicity. In the Metropolis-Hastings Markov chains
we have two important particular advantages, sourced by the nature of the chain. Such a chain is always reversible and the target distribution stands
for the (known) stationary distribution. These facts simplify the environment of the proof scenario.

	The most famous convergence result for the MH-chains, formulated in easily verifiable conditions, is announced for example in \cite{basic-3} (Theorem 7.4 along with Corollary 7.5).
Therein is shown that the positivity of the proposal distribution
\begin{equation}\label{ppc}
    q(x|x')>0 \:\: for\: all \:\: (x,x') \in \mathbb{X}\times\mathbb{X},
\end{equation}
provides the irreducibility of the corresponding MH-chain and also if the algorithm admits the event $x_{(n)}=x_{(n-1)}$ with nonzero probability
then the MH-chain is aperiodic.
The second claim of Theorem 7.4 in \cite{basic-3} says that both irreducibility and aperiodicity imply the total variation convergence.
Also in \cite{basic-3} is pointed out the existence of classes of examples in which the event $x_{(n)}=x_{(n-1)}$ never occurs.

Concisely formulated our main practical result (Corollary \ref{coroll42})
states that if the proposal density $q(\cdot|\cdot)$ is positive then (\ref{e1-3}) holds
regardless from the shape of the initial density.
Therefore (\ref{ppc}) provides also the aperiodicity of the chain because it is a necessary condition for an irreducible chain to converge.

\section{The $L^2(\pi)$ structure}
Following Stroock \cite{origin-3} we shall consider the Hilbert space $L^2(\pi)$ with an inner product
\begin{equation*}
    \langle f,g\rangle_{\pi} = \int f(x) g(x) \pi(x) \lambda(dx).
\end{equation*}
The space $L^2(\pi)$ consists of the measurable functions $f(\cdot):\mathbb{X}\rightarrow\mathbb{\bar{R}}$ for which
\begin{equation*}
    \| f\|_{2,\pi} = \sqrt{\int |f(x)|^2 \pi(x)\lambda(dx)} < \infty
\end{equation*}
(see e.g. \cite{book-8,book-10,book-15}). Define the operator
\begin{multline}\label{e2-1}
    \mathcal{K}[f](x) = \int\kappa(x\rightarrow x')f(x')\lambda(dx') \\
    = \int\mathring{\kappa}(x\rightarrow x')f(x')\lambda(dx') + \phi(x)f(x)
\end{multline}
which is formally conjugate to the basic transition operator of the chain
\begin{multline}\label{e2-2}
    \mathcal{\hat{K}}[f](x') = \int f(x)\kappa(x\rightarrow x')\lambda(dx) \\
    = \int f(x)\mathring{\kappa}(x\rightarrow x')\lambda(dx) + \phi(x')f(x')
\end{multline}
where the sub-kernel $\mathring{\kappa}(\cdot \rightarrow \cdot):\mathbb{X}\times\mathbb{X}\rightarrow \mathbb{\bar{R}}_+$
\begin{equation*}
    \mathring{\kappa}(x\rightarrow x') = \min\left( q(x'|x),\frac{\pi(x')}{\pi(x)} q(x|x')\right)
\end{equation*}
is nonnegative $\mathcal{A}\times\mathcal{A}$ measurable function and the function $\phi(\cdot):\mathbb{X}\rightarrow \mathbb{\bar{R}}_+$
\begin{equation*}
    \phi(x) = \int \left( 1 - \alpha(x,z) \right) q(z|x) \lambda(dz)
\end{equation*}
is measurable with $0\leq \phi(x)\leq 1$ for $x\in\mathbb{X}$. Actually $\kappa(\cdot\rightarrow \cdot)$ stands for a transition kernel of the transition
operator $\mathcal{\hat{K}}$ and simply is a kernel of the conjugate operator $\mathcal{K}$.

Put $\kappa_1 (x\rightarrow x')=\kappa (x\rightarrow x')$ and compose formally the sequence of kernels
\begin{equation*}
    \kappa_n (x\rightarrow x') = \int \kappa_{n-1} (x\rightarrow z) \kappa_1 (z\rightarrow x') \lambda(dz), n = 2,3,\ldots,
\end{equation*}
which are just the transition kernels of the transition-like operators $\mathcal{\hat{K}}^n$ in a sense that
\begin{equation*}
    \mathcal{\hat{K}}^n[f](x') = \int f(x)\kappa_n (x\rightarrow x')\lambda(dx)
\end{equation*}
and the usual kernels of the operators $\mathcal{K}^n$, i.e.
\begin{equation*}
    \mathcal{K}^n[f](x) = \int\kappa_n (x\rightarrow x')f(x')\lambda(dx').
\end{equation*}
Put also $\mathring{\kappa}_1 (x\rightarrow x')=\mathring{\kappa} (x\rightarrow x')$ and compose the sub-kernels
\begin{equation}\label{e2-3}
    \mathring{\kappa}_n (x\rightarrow x') = \int \mathring{\kappa}_{n-1} (x\rightarrow z) \mathring{\kappa}_1 (z\rightarrow x') \lambda(dz), n = 2,3,\ldots.
\end{equation}
One can find by induction that
\begin{multline*}
    \mathcal{K}^n[f](x) = \int\kappa_n(x\rightarrow x')f(x')\lambda(dx') \\
    = \int\mathring{\kappa}_n(x\rightarrow x')f(x')\lambda(dx') + \int \chi_n (x\rightarrow x')f(x')\lambda(dx') +
    \phi^n(x)f(x)
\end{multline*}
where $\chi_n(\cdot \rightarrow \cdot):\mathbb{X}\times\mathbb{X}\rightarrow \mathbb{\bar{R}}_+$ is some nonnegative $\mathcal{A}\times\mathcal{A}$ measurable function.
One can verify that $\kappa_n (\cdot\rightarrow \cdot)$ also satisfies the detailed balance condition and the Chapman-Kolmogorov equation
\begin{equation*}
    \kappa_{m+n} (x\rightarrow x') = \int \kappa_{m} (x\rightarrow z) \kappa_{n} (z\rightarrow x') \lambda(dz), m = 1,2,\ldots, n = 1,2,\ldots,
\end{equation*}
and the same is true for the sub-kernel $\mathring{\kappa}_n (\cdot\rightarrow \cdot)$.

The proofs of the first two claims of the following proposition can be found substantially for example in \cite{ origin-3},
but here we propose our ones for the sake of completeness.

\begin{prop}\label{p21}
Suppose {\color{blue} $\mathcal{H}$} holds and let $f\in L^2(\pi)$.
Then the following assertions are valid for the operator defined in (\ref{e2-1}) . \newline
1) $\mathcal{K}[f]\in L^2(\pi)$ and also
\begin{equation}\label{e2-4}
    \| \mathcal{K}[f]\|_{2,\pi} \leq \| f \|_{2,\pi}.
\end{equation}
2) The operator $\mathcal{K}:L^2(\pi)\rightarrow L^2(\pi)$ is self-adjoint and for its norm we have
\begin{equation}\label{e2-5}
    \| \mathcal{K} \| \leq 1.
\end{equation}
3) Suppose that there exists an integer $n\geq 1$ such that
$\mathring{\kappa}_n (\cdot\rightarrow \cdot)>0$ a.e. $(\lambda\times\lambda)$ in $\mathbb{X}\times\mathbb{X}$
where $\mathring{\kappa}_n (\cdot\rightarrow \cdot)$ is a composite sub-kernel defined in (\ref{e2-3}).
Let also $h\in L^2(\pi)$ be a function for which $\mathcal{K}^{n}[h] = h$.
Then there exists a constant $\gamma$ such that $h(\cdot) = \gamma$ a.e. $(\lambda)$ in $\mathbb{X}$.
\end{prop}
\begin{proof}
1) Let $f\in L^2(\pi)$. Applying the Holder's inequality to the functions $|f|\sqrt\kappa$ and $\sqrt\kappa$ we find
\begin{equation*}
    \left| \mathcal{K}[f](x) \right|^2   \leq  \int \kappa(x\rightarrow x') |f(x')|^2 \lambda(dx').
\end{equation*}
Multiplying the latter with $\pi(x)$ and integrating over $\mathbb{X}$ we get
\begin{multline*}
    \| \mathcal{K}[f] \|_{2,\pi}^{2}
    \leq \int \left( \int \pi(x)\kappa(x\rightarrow x') |f(x')|^2 \lambda(dx') \right) \lambda(dx) \\
    = \int \left( \int \pi(x')\kappa(x'\rightarrow x) |f(x')|^2 \lambda(dx') \right) \lambda(dx) \\
     = \int \pi(x')|f(x')|^2 \lambda(dx') = \| f \|_{2,\pi}^{2}.
\end{multline*}
Here we use the detailed balance condition (\ref{e1-2}) and the Tonelli's theorem which allows
us to interchange the order of integration. Thus we prove simultaneously
the inequality (\ref{e2-4}) and the fact that $\mathcal{K}[f]\in L^2(\pi)$.

2) Let $f\in L^2(\pi)$ and $g\in L^2(\pi)$ and write by means of the Fubini's theorem and by the detailed balance condition
\begin{multline*}
    \langle\mathcal{K}[f],g\rangle_{\pi} = \int \left(\int\kappa(x\rightarrow x')f(x')\lambda(dx')\right) g(x) \pi(x) \lambda(dx) \\
    = \int \left(\int \pi(x) \kappa(x\rightarrow x')f(x')\lambda(dx')\right) g(x) \lambda(dx) \\
    = \int f(x') \left( \int \kappa(x'\rightarrow x) g(x) \lambda(dx) \right) \pi(x') \lambda(dx') = \langle f,\mathcal{K}[g]\rangle_{\pi}.
\end{multline*}
which proves that the operator $\mathcal{K}$ is self-adjoint.
The inequality (\ref{e2-5}) follows immediately from (\ref{e2-4}).

3) Write the identity
\begin{equation*}
    h^2(x') = h^2(x) + 2h(x)(h(x')-h(x)) + (h(x')-h(x))^2
\end{equation*}
multiply with $\kappa_{n}(x\rightarrow x')$ and integrate. Then we get
\begin{equation*}
    \int \kappa_{n}(x\rightarrow x') h^2(x') \lambda(dx') = h^2(x) + \int \kappa_{n}(x\rightarrow x') (h(x')-h(x))^2 \lambda(dx')
\end{equation*}
because
\begin{equation*}
    \int \kappa_{n}(x\rightarrow x') 2h(x)(h(x')-h(x)) \lambda(dx') = 2h(x)\left( \mathcal{K}^{n}[h](x) - h(x) \right) = 0.
\end{equation*}
Multiply with $\pi(x)$ and integrate. Then
\begin{multline}\label{e2-7}
    \int \left( \int \pi(x) \kappa_{n}(x\rightarrow x') h^2(x') \lambda(dx') \right) \lambda(dx) = \int \pi(x) h^2(x) \lambda(dx) \\
    + \int \pi(x) \left( \int \kappa_{n}(x\rightarrow x') (h(x')-h(x))^2 \lambda(dx') \right) \lambda(dx).
\end{multline}
It is easy to see that the left-hand side in (\ref{e2-7}) is equal to the first addend in the right-hand side. Therefore
\begin{equation*}
    \int \pi(x) \left( \int \kappa_{n}(x\rightarrow x') (h(x')-h(x))^2 \lambda(dx') \right) \lambda(dx) = 0
\end{equation*}
which implies immediately that also
\begin{equation*}
    \int \pi(x) \left( \int \mathring{\kappa}_{n}(x\rightarrow x') (h(x')-h(x))^2 \lambda(dx') \right) \lambda(dx) = 0.
\end{equation*}
Now the inequalities $\pi(\cdot)>0$ and
$\kappa_{n}(\cdot \rightarrow \cdot) > 0$ give that there exists a constant $\gamma$ such that $h(\cdot) = \gamma$ a.e. $(\lambda)$ in $\mathbb{X}$.
\end{proof}

\begin{prop}\label{p22}
Suppose {\color{blue} $\mathcal{H}$} holds and let $\mathcal{K}:L^2(\pi)\rightarrow L^2(\pi)$ be the operator defined above.
Then for any integer $\nu\geq 1$ the following assertions are valid. \newline
1) Every power $\mathcal{K}^{2\nu n}$, $n=1,2,\ldots$, is positive operator, i.e.
$\langle\mathcal{K}^{2\nu n}[h],h\rangle_{\pi}\geq 0$ for any $h\in L^2(\pi)$. \newline
2) The sequence $(\mathcal{K}^{2\nu n})$
is decreasing, i.e. $\langle\mathcal{K}^{2\nu n+2\nu}[h],h\rangle_{\pi}\leq \langle\mathcal{K}^{2\nu n}[h],h\rangle_{\pi}$
for any $h\in L^2(\pi)$, $n=1,2,\ldots$. \newline
3) All the operators $\left(\mathcal{K}^{2\nu n} - \mathcal{K}^{2\nu n+2\nu p}\right)$ for $n = 1,2,\ldots$ and $p = 1,2,\ldots$ are also positive.
\end{prop}
\begin{proof} 1) Let $h\in L^2(\pi)$. The operator $\mathcal{K}$ is self-adjoint therefore
\begin{equation*}
    \langle\mathcal{K}^{2\nu n}[h],h\rangle_{\pi} = \langle\mathcal{K}^{\nu n}[h],\mathcal{K}^{\nu n}[h]\rangle_{\pi} \geq 0.
\end{equation*}
2) We have $\|\mathcal{K}\|\leq 1$ therefore
\begin{multline*}
    \langle\mathcal{K}^{2\nu n+2\nu}[h],h\rangle_{\pi} = \langle\mathcal{K}^{\nu n+\nu}[h],\mathcal{K}^{\nu n+\nu}[h]\rangle_{\pi}  \\
    = \| \mathcal{K}^{\nu n+\nu}[h]\|_{2,\pi}^2 = \| \mathcal{K}^{\nu}\left[ \mathcal{K}^{\nu n}[h]\right] \|_{2,\pi}^2\leq  \| \mathcal{K}^{\nu n}[h]\|_{2,\pi}^2   \\
    = \langle\mathcal{K}^{\nu n}[h],\mathcal{K}^{\nu n}[h]\rangle_{\pi} = \langle\mathcal{K}^{2\nu n}[h],h\rangle_{\pi}.
\end{multline*}
3) The positiveness of the operator $\left(\mathcal{K}^{2\nu n} - \mathcal{K}^{2\nu n+2\nu p}\right)$ means that
\begin{equation*}
    \langle \left( \mathcal{K}^{2\nu n} - \mathcal{K}^{2\nu n+2\nu p}\right)[h],h \rangle_\pi \geq 0
\end{equation*}
for any $h\in L^2(\pi)$ that is equivalent to
\begin{equation*}
    \langle \mathcal{K}^{2\nu n}[h],h \rangle_\pi \geq \langle \mathcal{K}^{2\nu n+2\nu p}[h],h \rangle_\pi
\end{equation*}
which follows immediately from 2).
\end{proof}

Here we are at position to prove that the operator sequence $(\mathcal{K}^n)$ has a strong limit.
More precisely we are going to prove that for every $f\in L^2(\pi)$ there exists the limit
\begin{equation*}
    \lim_{n\rightarrow\infty} \mathcal{K}^n [f] = \langle f,\mathbf{1}\rangle_{\pi}\mathbf{1}
\end{equation*}
where $\mathbf{1}$ denotes the constant function which equals to one.

Further we shall need the following condition of positiveness.

{\color{blue} $\mathcal{H}_p$:}
Assume that there exists an integer $\nu\geq 1$ for which $\mathring{\kappa}_\nu (\cdot \rightarrow \cdot)>0$ a.e. $(\lambda\times\lambda)$ in $\mathbb{X}\times\mathbb{X}$. $\Box$

The condition {\color{blue} $\mathcal{H}_p$} is met for example when the proposal density is positive. Remember that
the target density is positive by condition {\color{blue} $\mathcal{H}$}.

\begin{thm}\label{t21} Suppose {\color{blue} $\mathcal{H}$} and {\color{blue} $\mathcal{H}_p$} hold. Then for every $f\in L^2(\pi)$ we have
\begin{equation}\label{e2-8}
    \lim_{n\rightarrow\infty} \| \mathcal{K}^n [f] - \langle f,\mathbf{1}\rangle_{\pi}\mathbf{1}\|_{2,\pi}=0.
\end{equation}
\end{thm}
\begin{proof}
It is not difficult to find out that for any real Hilbert space $H$ with an inner product $\langle\cdot,\cdot\rangle$ and a norm
$\|u\| = \sqrt{\langle u,u\rangle}, u\in H$, with a given linear bounded self-adjoint positive operator $T:H\rightarrow H$ it holds the inequality
\begin{equation}\label{e2-9}
    \|Tu\|^2\leq\|T\|\langle Tu,u\rangle.
\end{equation}
The proof of (\ref{e2-9}) will be given at the end of the paper.
Choose arbitrary $f\in L^2(\pi)$. Applying (\ref{e2-9}) to the positive operators $(\mathcal{K}^{2\nu n} - \mathcal{K}^{2\nu n+2\nu p})$
for $n = 1,2,\ldots$ and $p = 1,2,\ldots$ we get
\begin{equation*}
    \| (\mathcal{K}^{2\nu n} - \mathcal{K}^{2\nu n+2\nu p})[f]\|_{2,\pi}^2 \leq \| \mathcal{K}^{2\nu n} - \mathcal{K}^{2\nu n+2\nu p} \| \langle\mathcal{K}^{2\nu n} [f] - \mathcal{K}^{2\nu n+2\nu p}[f],f\rangle_{\pi}
\end{equation*}
from which follows that
\begin{equation}\label{e2-10}
        \| \mathcal{K}^{2\nu n} [f] - \mathcal{K}^{2\nu n+2\nu p}[f]\|_{2,\pi}^2 \leq
        2\left( \langle\mathcal{K}^{2\nu n} [f],f\rangle_{\pi} - \langle\mathcal{K}^{2\nu n+2\nu p}[f],f\rangle_{\pi} \right).
\end{equation}
From Proposition \ref{p22} we know that the numerical sequence $(\langle\mathcal{K}^{2\nu n} [f],f\rangle_{\pi})_{n=1}^\infty$
is decreasing an bounded from below by zero therefore this sequence is convergent. Now from (\ref{e2-10}) it follows that the sequence of the
powers $(\mathcal{K}^{2\nu n} [f])_{n=1}^\infty$ is a Cauchy sequence in $L^2(\pi)$ therefore it has a limit $h\in L^2(\pi)$ for
which obviously it holds $\mathcal{K}^{2\nu}[h] = h$. From Proposition \ref{p21}(3) (with $n=2\nu$) we get
that $h(\cdot) = \gamma$ a.e. $(\lambda)$ in $\mathbb{X}$ with some constant $\gamma$ because
the stated positiveness of the sub-kernel $\mathring{\kappa}_{\nu}(\cdot\rightarrow \cdot)$ in {\color{blue} $\mathcal{H}_p$} provides that
\begin{equation*}
    \mathring{\kappa}_{2\nu}(x\rightarrow x') = \int \mathring{\kappa}_{\nu}(x\rightarrow z)\mathring{\kappa}_{\nu}(z\rightarrow x') \lambda(dz) > 0
\end{equation*}
a.e. $(\lambda\times\lambda)$ in $\mathbb{X}\times\mathbb{X}$.
We have $\mathcal{K}^{2\nu n} [f]\rightarrow \gamma\mathbf{1}$ whence
$\langle\mathcal{K}^{2\nu n} [f],\mathbf{1}\rangle_{\pi}\rightarrow \gamma \langle\mathbf{1},\mathbf{1}\rangle_{\pi} = \gamma$ which gives
\begin{equation*}
    \langle f,\mathbf{1}\rangle_{\pi} = \langle f,\mathcal{K}^{2\nu n}[\mathbf{1}]\rangle_{\pi} =
    \langle\mathcal{K}^{2\nu n} [f],\mathbf{1}\rangle_{\pi}\rightarrow \gamma
\end{equation*}
therefore $\gamma = \langle f,\mathbf{1}\rangle_{\pi}$ which proves (\ref{e2-8}) for the subsequence of the powers $(2\nu n)_{n=1}^\infty$, i.e. that
\begin{equation}\label{e2-11}
    \lim_{n\rightarrow\infty} \mathcal{K}^{2\nu n} [f] = \langle f,\mathbf{1}\rangle_{\pi}\mathbf{1}.
\end{equation}
From (\ref{e2-11}) we obtain
\begin{multline}\label{e2-12}
        \lim_{n\rightarrow\infty} \mathcal{K}^{2\nu n+m} [f] = \lim_{n\rightarrow\infty} \mathcal{K}^{2\nu n} \left[\mathcal{K}^{m}[f]\right]  \\
        = \langle \mathcal{K}^{m}[f],\mathbf{1}\rangle_{\pi}\mathbf{1} = \langle f,\mathcal{K}^{m}[\mathbf{1}]\rangle_{\pi}\mathbf{1}
        = \langle f,\mathbf{1}\rangle_{\pi}\mathbf{1}
\end{multline}
which proves (\ref{e2-8}) for the all the power subsequences $(2\nu n + m)_{n=1}^\infty$ where $m$, $1\leq m < 2\nu$, is a nonzero remainder
after a division by $2\nu$. Now it is not difficult to see that the validity of (\ref{e2-8}) follows from (\ref{e2-11}) and (\ref{e2-12}).
\end{proof}

\section{Convergence with respect to $TV$ distance}
Our main purpose is to investigate the behavior of the operator sequence $(\mathcal{\hat{K}}^n)$
rather than the sequence $(\mathcal{K}^n)$ where the transition operator $\mathcal{\hat{K}}$ is
defined in (\ref{e2-2}) because it actually corresponds to the Markov chain.
Note that if $\mu_1$ and $\mu_2$ are absolutely continuous probability measures (w.r.t. $\lambda$) with densities $f_1(\cdot)$ and $f_2(\cdot)$
then for the total variation distance $d_{TV}(\mu_1,\mu_2)$ it holds (see e.g. \cite{paper-2,paper-1})
\begin{equation*}
    d_{TV}(\mu_1,\mu_2) = \frac{1}{2} \int |f_1(x)-f_2(x)| \lambda(dx).
\end{equation*}

Let $L_\mathbb{X}^1$ be the Banach space of the measurable functions $f(\cdot):\mathbb{X}\rightarrow\mathbb{\bar{R}}$ for which
\begin{equation*}
    \|f\|_1 = \int |f(x)| \lambda(dx) < \infty
\end{equation*}
provided with the usual norm $\|\cdot\|_1$. We have
\begin{equation}\label{e3-1}
    \|f\|_1 \leq \|f/\pi\|_{2,\pi}
\end{equation}
because (by Cauchy-Schwarz inequality)
\begin{multline*}
    \|f\|_1 = \int |f(x)| \lambda(dx)  \\
    = \int \pi(x) \left|\frac{f(x)}{\pi(x)}\right| \lambda(dx) = \int \sqrt{\pi(x)} \left|\sqrt{\pi(x)}\frac{f(x)}{\pi(x)}\right| \lambda(dx)  \\
    \leq \sqrt{\int\pi(x)\lambda(dx)} \sqrt{\int\pi(x)\left|\frac{f(x)}{\pi(x)}\right|^2  \lambda(dx)} = \|f/\pi \|_{2,\pi}.
\end{multline*}

From (\ref{e3-1}) it follows that if $f/\pi\in L^2(\pi)$ then $f\in L_{\mathbb{X}}^1$.

\begin{prop}\label{p31}
Suppose {\color{blue} $\mathcal{H}$} and {\color{blue} $\mathcal{H}_p$} hold.
Let $f(\cdot)$ be a function such that $f/\pi\in L^2(\pi)$ and put $\gamma = \int f(x) \lambda(dx)$. Then
\begin{equation}\label{e3-2}
    \lim_{n\rightarrow\infty} \| \mathcal{\hat{K}}^n [f] - \gamma{\pi}\|_1 = 0.
\end{equation}
\end{prop}
\begin{proof}
First of all notice that (\ref{e3-1}) guarantees the existence of the constant $\gamma$.
It can be shown by induction that
\begin{equation*}
    \mathcal{\hat{K}}^n [f] (x') = \pi(x')\mathcal{K}^n \left[\frac{f}{\pi}\right](x'), n = 1,2,\ldots,
\end{equation*}
whence by means of the Cauchy-Schwarz inequality we obtain
\begin{multline}\label{e3-3}
    \| \mathcal{\hat{K}}^n [f] - \gamma{\pi}\|_1 = \int | \mathcal{\hat{K}}^n [f] (x') - \gamma\pi(x') | \lambda(dx')  \\
    = \int \pi(x') \left|\mathcal{K}^n \left[\frac{f}{\pi}\right](x') - \gamma\mathbf{1}(x') \right| \lambda(dx')  \\
    = \int \sqrt{\pi(x')} \left( \sqrt{\pi(x')} \left|\mathcal{K}^n \left[\frac{f}{\pi}\right](x') - \gamma\mathbf{1}(x') \right| \right) \lambda(dx')  \\
    \leq \sqrt{\int \pi(x') \left|\mathcal{K}^n \left[\frac{f}{\pi}\right](x') - \gamma\mathbf{1}(x') \right|^2 \lambda(dx')}
    = \left\| \mathcal{K}^n \left[\frac{f}{\pi}\right] - \gamma\mathbf{1} \right\|_{2,\pi}
\end{multline}
On the other hand by (\ref{e2-8}) it follows that
\begin{equation}\label{e3-4}
    \lim_{n\rightarrow\infty} \left\| \mathcal{K}^n \left[\frac{f}{\pi}\right] - \gamma\mathbf{1} \right\|_{2,\pi} = 0
\end{equation}
because
\begin{equation*}
    \left\langle \frac{f}{\pi},\mathbf{1}\right\rangle_{\pi} = \int \frac{f(x)}{\pi(x)} \pi(x) \lambda(dx) = \int f(x) \lambda(dx) = \gamma.
\end{equation*}
Now the validity of (\ref{e3-2}) follows immediately from (\ref{e3-3}) and (\ref{e3-4}).
\end{proof}

Hereafter we shall prepare for the final results. Put
\begin{equation*}
    \mathbb{X}_m = \left( x\in\mathbb{X} | \pi(x) \geq \frac{1}{m} \right), m =1,2,\ldots.
\end{equation*}
Obviously $\mathbb{X}_{m}\subseteq\mathbb{X}_{m+1}$ and $\mathbb{X}=\cup_{m=1}^\infty\mathbb{X}_m$.
For $f(\cdot):\mathbb{X}\rightarrow \bar{\mathbb{R}}$ put $f_{[m]}(x) = f(x)$ where $x\in\mathbb{X}_m$ and $f_{[m]}(x) = 0$ elsewhere.

The following proposition helps to prove Theorem \ref{t31}.
\begin{prop}\label{p33} Suppose {\color{blue} $\mathcal{H}$} holds. Then the following assertions are true. \newline
1) Let $f\in\ L_\mathbb{X}^1$. Then $\mathcal{\hat{K}}[f]\in\ L_\mathbb{X}^1$ and
\begin{equation}\label{e3-5}
    \| \mathcal{\hat{K}}[f] \|_1 \leq \|f\|_1
\end{equation}
consequently for any $f\in L_\mathbb{X}^1$ and $g\in L_\mathbb{X}^1$ and any $n=1,2,\ldots$ we have
\begin{equation}\label{e3-6}
    \| \mathcal{\hat{K}}^n[f] - \mathcal{\hat{K}}^n[g] \|_1 \leq \|f-g\|_1.
\end{equation}
2) Let $f\in L_\mathbb{X}^1$ be a bounded function. Then $f_{[m]}/\pi\in L^2(\pi)$ and
\begin{equation}\label{e3-7}
    \lim_{m\rightarrow\infty} \|f-f_{[m]}\|_1 = 0.
\end{equation}
3) Let $f\in L_\mathbb{X}^1$ and put
\begin{equation*}
    \gamma = \int f(x) \lambda(dx), \gamma_{[m]} = \int f_{[m]}(x) \lambda(dx), m = 1,2,\ldots.
\end{equation*}
Then
\begin{equation}\label{e3-8}
    \| \gamma_{[m]}\pi - \gamma\pi \|_1 \leq \|f-f_{[m]}\|_1.
\end{equation}
\end{prop}
\begin{proof} 1) We have
\begin{equation*}
    |\mathcal{\hat{K}}[f](x')| \leq \int |f(x)|\kappa(x\rightarrow x')\lambda(dx)
\end{equation*}
whence (again by means of the Fubini's theorem)
\begin{multline*}
    \| \mathcal{\hat{K}}[f] \|_1  = \int |\mathcal{\hat{K}}[f](x')|\lambda(dx')  \\
    \leq \int \left( \int |f(x)|\kappa(x\rightarrow x') \lambda(dx) \right) \lambda(dx')  = \int |f(x)|\lambda(dx) = \|f\|_1
\end{multline*}
which proves (\ref{e3-5}).
The validity of (\ref{e3-6}) follows immediately from (\ref{e3-5}) and the linearity of $\mathcal{\hat{K}}$. \newline
2) According to the assumption $f(\cdot)$ is bounded consequently for some constant $C$ it holds $|f_{[m]}(x)| \leq C$ for $x\in\mathbb{X}_m$.
Then
\begin{equation*}
    \left| \frac{f_{[m]}(x)}{\pi(x)}\right| \leq Cm, x\in\mathbb{X}_m,
\end{equation*}
therefore
\begin{equation*}
    \int \pi(x) \left| \frac{f_{[m]}(x)}{\pi(x)}\right|^2 \lambda(dx) \leq \int_{\mathbb{X}_m} \pi(x) \left| Cm\right|^2 \lambda(dx) \leq C^2 m^2 < \infty
\end{equation*}
which proves that $f_{[m]}/\pi\in L^2(\pi)$. By the definition
\begin{equation*}
    \|f-f_{[m]}\|_1 = \int |f(x) - f_{[m]}(x)| \lambda(dx) = \int_{\mathbb{X}\backslash\mathbb{X}_m} |f(x)| \lambda(dx)
\end{equation*}
which proves (\ref{e3-7}) because $\mathbb{X}_m\nearrow\mathbb{X}$. \newline
3) We have
\begin{multline*}
    \| \gamma_{[m]}\pi - \gamma\pi\|_1 = \int | \gamma_{[m]}\pi(x') - \gamma\pi(x') | \lambda(dx')   \\
    = |\gamma_{[m]} - \gamma| \int\pi(x')\lambda(dx') = \left| \int_{\mathbb{X}\backslash\mathbb{X}_m} f(x) \lambda(dx)\right|
    \leq \| f - f_{[m]}\|_1
\end{multline*}
which proves (\ref{e3-8}).
\end{proof}

We are ready to give more general conditions under which (\ref{e3-2}) is valid.

\begin{thm}\label{t31} Suppose {\color{blue} $\mathcal{H}$} and {\color{blue} $\mathcal{H}_p$} hold.
Let $f\in L_{\mathbb{X}}^1$ and put $\gamma = \int f(x) \lambda(dx)$. Then
\begin{equation}\label{e3-9}
    \lim_{n\rightarrow\infty} \| \mathcal{\hat{K}}^n [f] - \gamma{\pi}\|_1 = 0.
\end{equation}
Therefore if $f\in L_{\mathbb{X}}^1$ is a probability density function (w.r.t. $\lambda$) on $\mathbb{X}$ then
\begin{equation}\label{e3-10}
    \lim_{n\rightarrow\infty} d_{TV}( \mu[ \mathcal{\hat{K}}^n [f] ] , \mu[\pi] ) = 0.
\end{equation}
\end{thm}
\begin{proof} In the beginning of this proof we shall assume that the function $f\in L_{\mathbb{X}}^1$ is bounded. Put again
\begin{equation*}
    \gamma_{[m]} = \int f_{[m]}(x) \lambda(dx), m=1,2,\ldots.
\end{equation*}
For any $n=1,2,\ldots$ and $m=1,2,\ldots$ we can write
\begin{equation}\label{e3-11}
    \| \mathcal{\hat{K}}^n [f] - \gamma\pi \|_1 \leq \| \mathcal{\hat{K}}^n [f] - \mathcal{\hat{K}}^n [f_{[m]}] \|_1 + \\
    \| \mathcal{\hat{K}}^n [f_{[m]}] - \gamma_{[m]}\pi \|_1 + \| \gamma_{[m]}\pi - \gamma\pi\|_1
\end{equation}
Choose some $\varepsilon>0$. By (\ref{e3-7}) fix an integer $m\geq 1$ such that $\| f-f_{[m]}\|_1<\varepsilon/3$.
Then according to (\ref{e3-6}) we obtain
\begin{equation}\label{e3-12}
    \| \mathcal{\hat{K}}^n [f] - \mathcal{\hat{K}}^n [f_{[m]}] \|_1 < \frac{\varepsilon}{3}
\end{equation}
for any $n=1,2,\ldots$ and according to (\ref{e3-8}) we obtain
\begin{equation}\label{e3-13}
    \| \gamma_{[m]}\pi - \gamma\pi\|_1 < \frac{\varepsilon}{3}.
\end{equation}
For such a fixed $m$ we have from Proposition \ref{p33}(2) that $f_{[m]}/\pi\in L^2(\pi)$ therefore by Proposition \ref{p31} we get that
\begin{equation*}
    \lim_{n\rightarrow\infty} \|\mathcal{\hat{K}}^n[f_{[m]}] - \gamma_{[m]} \pi \|_1 = 0
\end{equation*}
consequently we can choose an positive integer $n_0$ such that
\begin{equation}\label{e3-14}
    \|\mathcal{\hat{K}}^n[f_{[m]}] - \gamma_{[m]} \pi \|_1 < \frac{\varepsilon}{3}
\end{equation}
for any $n>n_0$. Replacing the inequalities (\ref{e3-12}), (\ref{e3-13}) and (\ref{e3-14}) in (\ref{e3-11}) we receive that
\begin{equation*}
    \| \mathcal{\hat{K}}^n [f] - \gamma\pi \|_1 < \frac{\varepsilon}{3} + \frac{\varepsilon}{3} + \frac{\varepsilon}{3} = \varepsilon
\end{equation*}
for any $n>n_0$ which by definition proves the validity of (\ref{e3-9}) for the case of bounded $f\in L_{\mathbb{X}}^1$.

Choose now arbitrary $f\in L_{\mathbb{X}}^1$ and put $f_m(x)=f(x)$ where $|f(x)|\leq m$ and $f_m(x)=0$ elsewhere, $m=1,2,\ldots$.
In the same way as for (\ref{e3-11}) one can get
\begin{equation*}
    \| \mathcal{\hat{K}}^n [f] - \gamma\pi \|_1 \leq \| \mathcal{\hat{K}}^n [f_m] - \gamma_m\pi \|_1 + 2\| f_m - f \|_1
\end{equation*}
where
\begin{equation*}
    \gamma_m = \int f_m(x)\lambda(dx).
\end{equation*}
From the first part of the proof we already know that for arbitrary fixed positive integer $m$ it holds
\begin{equation*}
    \lim_{n\rightarrow\infty} \| \mathcal{\hat{K}}^n [f_m] - \gamma_m\pi \|_1 = 0
\end{equation*}
because $f_m(\cdot)$ is bounded.
Now the validity of (\ref{e3-9}) is a consequence of the well-known fact that
\begin{equation*}
    \lim_{m\rightarrow\infty}\| f_m - f \|_1=0.
\end{equation*}

The validity of (\ref{e3-10}) follows immediately from (\ref{e3-9}).
\end{proof}

Theorem \ref{t31} allows us to prove the following practically valuable result with easily verifiable conditions.

\begin{coroll}\label{coroll42}
Suppose {\color{blue} $\mathcal{H}$} holds.
Let $q(x|x')>0$
a.e. $(\lambda\times\lambda)$ in $\mathbb{X}\times\mathbb{X}$.
Then (\ref{e1-3}) is valid for every initial density $f_{0}(\cdot)\in L_{\mathbb{X}}^1$.
\end{coroll}
\begin{proof}
By conditions we have $\pi(x)>0$ for all $x\in\mathbb{X}$ and $q(\cdot|\cdot)>0$ a.e. $(\lambda\times\lambda)$ in $\mathbb{X}\times\mathbb{X}$.
These inequalities guarantee that $\mathring{\kappa}(\cdot\rightarrow \cdot)>0$ a.e. $(\lambda\times\lambda)$ in $\mathbb{X}\times\mathbb{X}$
therefore we can use {\color{blue} $\mathcal{H}_p$} with $\nu = 1$. At this point we can apply Theorem \ref{t31} which
implies consequently the validity of (\ref{e3-9}), (\ref{e3-10}) and (\ref{e1-3}).
\end{proof}

Now we can conclude that the condition (\ref{ppc}) (besides the irreducibility shown for example in \cite{basic-3}) provides also the aperiodicity of the chain, because it is a necessary convergence condition for the irreducible chains (as it follows from Theorem 5.4.4 \cite{advanced-2}).

\section{Remarks}
In this paper we offer a possible way to prove the total variance convergence of MH-chains, under the kernel positivity condition, by means referred more to the
functional analysis rather than to the classical probability constructions. The offered method is backgrounded by some schemes described in \cite{origin-3} and does not utilize
the notions of the irreducibility and aperiodicity. However it has an essential drawback, because the overall proof scheme lays heavily on the detailed balance
condition property of the MH-chains. Thus our approach basically cannot apply to MCMC algorithms
which do not fulfill the detailed balance condition.

It is not difficult to prove that the Kullback-Leibler divergence
between the current MH-chain density $f_{(n)}(\cdot)$ and the target distribution $\pi(\cdot)$
do not increase as $n\to\infty$. The same is valid for the $d_{TV} (\mu[f_{(n)}],\mu[\pi])$, which is shown for example
in \cite{advanced-2} (Proposition 13.3.2). The latter can be considered as a good testimony for the validity of (\ref{e1-3}) at all under some natural requirements.

For the sake of completeness let us prove the validity of (\ref{e2-9}). The operator $T$ is linear bounded self-adjoint and positive
in the real Hilbert space $H$ therefore it holds the Cauchy-like inequality
\begin{equation*}
    \left|\langle Tu,v\rangle\right|^2 \leq \langle Tu,u\rangle \langle Tv,v\rangle
\end{equation*}
for all $u\in H$ and all $v\in H$. Putting in the latter $v = Tu$ we obtain
\begin{equation}\label{e5-1}
    \left|\langle Tu,Tu\rangle\right|^2 = \|Tu\|^4 \leq \langle Tu,u\rangle \langle T^2u,Tu\rangle.
\end{equation}
Now applying the classical Cauchy inequality and the inequality for the norm we get
\begin{equation*}
    \langle T^2u,Tu\rangle\leq \|T^2 u\| \|Tu\| \leq \|T\| \|Tu\| \|Tu\| = \|T\| \|Tu\|^2.
\end{equation*}
Replacing the latter in (\ref{e5-1}) we get the inequality
\begin{equation*}
    \|Tu\|^4 \leq \langle Tu,u\rangle \|T\| \|Tu\|^2
\end{equation*}
that is equivalent to $\|Tu\|^2 \leq \|T\| \langle Tu,u\rangle$, i.e. (\ref{e2-9}). Perhaps various proofs of (\ref{e2-9})
can be found in other places but we present here an explicit proof taking into account the importance of
this inequality in our construction.

The operator $\mathcal{\hat{K}}$ considered in (\ref{e2-1}) represents an interesting nontrivial example for a self-adjoint operator.
Also a careful inspection of the proofs in the discrete case leads to considering of a very interesting example of positive definite matrices
of the type $A=(\min (\alpha_i,\alpha_j))$ where $(\alpha_k)$ are mutually different positive numbers.

Finally note that in the general case (\ref{ppc}) is the only easily verifiable condition which provides the assumption {\color{blue} $\mathcal{H}_p$}.

\section*{Acknowledgements}
The authors are grateful to the referee for the constructive remarks on this work.

\end{document}